\documentclass[12pt,a4paper]{article}

\usepackage[T1]{fontenc} 		%font is compiled in a better way
\usepackage[english]{babel}
\usepackage[utf8]{inputenc}
\usepackage{amssymb}			%for leqslant, geqslant for example
\usepackage{amsmath}			%for mathematical symbols
\usepackage{amsfonts}			%for mathematical fonts, mathbb for example
\usepackage{amsthm}				%for \theoremstyle for example
\usepackage{indentfirst}		%automatic indents in new sections
\usepackage{titlesec} 			%to change style of section title
\usepackage[all]{xy}			%diagrams
\usepackage[top=2cm, bottom=2cm, left=2cm, right=2cm]{geometry}
\usepackage{hyperref}

\DeclareMathOperator{\Hom}{Hom}
\DeclareMathOperator{\Der}{Der}
\DeclareMathOperator{\Alg}{Alg}
\DeclareMathOperator{\Id}{Id}
\DeclareMathOperator{\Mod}{Mod}
\DeclareMathOperator{\Span}{Span}
\DeclareMathOperator{\Ima}{Im}
\DeclareMathOperator{\ev}{ev}

\titleformat*{\section}{\bf \large}

\providecommand{\keywords}[1]{\textbf{Keywords:} #1}
\providecommand{\msc}[1]{\textbf{Mathematics Subject Classification (2010):} #1}

\theoremstyle{plain}
\newtheorem{theorem}{Theorem}[section]
\newtheorem{lemma}[theorem]{Lemma}
\newtheorem{prop}[theorem]{Proposition}
\newtheorem{corollary}[theorem]{Corollary}
\theoremstyle{definition}
\newtheorem*{definition*}{Definition}

\newcommand{\sslash}{\mathbin{/\mkern-6mu/}}
\newcommand{\comment}[1]{\sslash #1 \sslash}

\numberwithin{equation}{section} %number equation with sections

\title{An algebraic approach to the ellipticity \\ of linear differential operators}
\date{}
\author{Sławomir Kapka\thanks{Faculty of Mathematics and Computer Science, University of Lodz, ul. Banacha 22, 90-238 Łódź, Poland. \newline e-mail: \href{mailto:slawomirkapka@gmail.com}{\texttt{slawomirkapka@gmail.com}}, ORCID: \href{https://orcid.org/0000-0003-3480-7294}{\texttt{orcid.org/0000-0003-3480-7294}} }}

\begin{document}

\maketitle

\begin{abstract}
We demonstrate a method of associating the principal symbol at a $K$-point with a linear differential operator acting between modules over a commutative algebra, and we use it to define the ellipticity of a linear differential operator in a purely algebraic way. We prove that the ellipticity is preserved by a surjective homomorphism of algebras. As an example, we show that for every real affine variety there is an elliptic linear differential operator acting on the algebra of regular functions on this variety.
\end{abstract}

\noindent\keywords{elliptic operators, K-points, principal symbol, universal derivations}

\medskip
\noindent\msc{13N05, 13N15, 58J05}

\section{Introduction}

\noindent\textbf{Motivation and main results}
\medskip

Differential calculus over commutative algebras is a perfect example of the abstarction in mathematics. It is a purely algebraic theory which evolved from the theory of linear differential operators (LDOs). \cite[sec. 16.8]{Grothendieck}, \cite[sec. 2]{H-S}, \cite{Vino} are the first papers, independent of each other, which studied LDOs from commutative algebra point of view, and which came up with the very same definition of an abstract LDO. This in itself already shows the naturality of such approach. After all, A. M. Vinogradov, the author of \cite{Vino}, and his co-workers pushed this idea further and eventually developed what we now call differential calculus over commutative algebras. 

The main object of study of this theory are abstract LDOs acting between arbitrary modules over a commutative algebra. As we can see from this preview \cite{Kra} of differential calculus over commutative algebras, this theory mostly focuses on functorial properties of the whole spaces of LDOs, rather then on individual LDOs. And from the perspective of differential geometry, this algebraic theory treats LDOs as global objects neglecting the pointwise behaviour. However, from the theory of LDOs on smooth manifolds, we know that there is a way to study the character of an LDO at a single point. In fact, it is the principal symbol that allows us to determine the pointwise behaviour of an LDO acting between global sections of smooth vector bundles. It is natural to ask if this rather analytic tool generalises to a more abstract algebraic framework. In this paper we give a positive answer to this question by demonstrating a method for studying the "pointwise" behaviour of LDOs acting between modules over a commutative algebra. In a nutshell, we follow ideas from \cite[ch. 9]{Nestruev}, and with a given $K$-algebra $A$ we consider its $K$-spectrum, which contains $K$-algebra homomorphisms from $A$ to $K$. Elements of a $K$-spectrum, called $K$-points, are good algebraic analogues of points, and thus it allows us to study abstract LDOs at $K$-points.

The main aim of this paper is to show that it is possible to study pointwise behaviour of linear differential operators already at the level of differential calculus over commutative algebras (Section \ref{sec_Elliptic}). More precisely, we show that the notion of principal symbol at a point and the ellipticity of an LDO can be formulated in a purely algebraic way (Theorem \ref{thm_El_Symbol} and the definition that follows). As an example, we use this to demonstrate that elliptic linear differential operators can be considered beyond the realm of smooth manifolds. First we prove that the ellipticity is preserved by a surjective homomorphism of algebras (Theorem \ref{thm_Af_InducElip}), and then we show that for every real affine variety there is an elliptic linear differential operator acting on the algebra of regular functions on this variety (Corollary \ref{cor_Af_ellipticOnVariety}).

\bigskip
\noindent\textbf{Sections preview}
\medskip

We devote Section \ref{sec_Deriv} to develop a preliminary algebraic formalism necessary for the main results from Sections \ref{sec_Examples}, \ref{sec_Elliptic} and \ref{sec_AffVar}. We introduce the notion of $\phi$-derivation, being here just a slightly "twisted" derivation, and we study its basic functorial properties. We focus on universal $\phi$-derivations, which we call $\phi$-differentials, and which are analogous to Kähler differentials from commutative algebra. 

In Section \ref{sec_Models} we present constructions of three notably different models of the universal $\phi$-derivation, and in Section \ref{sec_Examples} we demonstrate that these models appear naturally in differential calculus over commutative algebras and in both differential and algebraic geometry. More specifically, we show that for any $K$-algebra $A$ and any $K$-point $h$, the map $d_h:A\to T^*_h A$ is a $h$-differential (Proposition \ref{prop_Ex_Alg}). As an example, we demonstrate that the pair consisting of the Zarisiki cotangent space at a point and the differential at a point, defined on an affine variety, forms a universal $\phi$-derivation for $\phi$ being an evaluation at a point (Proposition \ref{prop_Ex_AG}). Similarly, we show that the pair consisting of the cotangent space at a point and the ordinary differential at a point, defined on a smooth manifold, is a universal $\phi$-derivation for $\phi$ being an evaluation at a point (Proposition \ref{prop_Ex_DG}).

We begin Section \ref{sec_Elliptic} with a brief comparison of the algebraic approach to linear differential operators with the classical analytic one. We focus on the fact that the former approach generalises the latter. Next, we have a look at the principal symbol of LDOs acting between sections of smooth vector bundles and we investigate possible ways to generalise it to LDOs acting between arbitrary modules. Eventually, we develop auxiliary Lemmas \ref{lem_El_Commute}, \ref{lem_El_TensorLike} and present our main Theorem \ref{thm_El_Symbol}, which states that there exists an algebraic analogue of the principal symbol at a point. As a corollary, we demonstrate how to define the ellipticity of a linear differential operator in a purely algebraic way.

In Section \ref{sec_AffVar} we study the behaviour of the principal symbol at a point of LDOs acting on algebras of regular functions on affine varieties. Given an affine variety, we have a natural surjective homomorphism from the algebra of polynomials to the algebra of regular functions on this variety. Thus we study the principal symbol of LDOs induced by algebra homomorphisms (Lemmas \ref{lem_Af_SymCom} and \ref{lem_Af_Symbol}), and we prove that ellipticity is preserved by surjective homomorphisms of algebras (Theroem \ref{thm_Af_InducElip}). Finally, we demonstrate that for every real affine variety there is an elliptic linear differential operator acting on the algebra of regular functions on this variety (Corollary \ref{cor_Af_ellipticOnVariety}).

We add Appendix \ref{app_Kähler} to clarify the relation between $\phi$-differentials and Kähler differentials. We show how to modify Kähler differentials to obtain $\phi$-differentials (Proposition \ref{prop_Apppen}).

\section{Universal $\phi$-derivations}
\label{sec_Deriv} 

The notion of derivation has already been established firmly in commutative algebra and algebraic geometry (see \cite[ch. 16]{Eisenbud} for example). A derivation is just any linear map from a commutative algebra to a module which satisfies Leibniz's rule. Furthermore, in \cite[p. 148]{Greub} we have the notion of $\phi$-derivations which satisfy a slightly "twisted" Leibniz's rule and which take values in algebras. We begin by merging these two concepts to obtain $\phi$-derivations which take values in modules.

Assume that $K$ is a fixed commutative and unitial ring. We constrain ourselves to commutative, unitial, associative $K$-algebras and unitial $K$-algebra homomorphisms. Let $A,B$ be two $K$-algebras and let $\phi:A\to B$ be a $K$-algebra homomorphism. Henceforth, this algebraic framework will be silently assumed unless we specify otherwise. 

\begin{definition*}
Let $M$ be a $B$-module. A $K$-linear map $D:A\to M$ will be called a \emph{$\phi$-derivation with values in $M$} (or just \emph{$\phi$-derivation}) if for every $f,g\in A$
\begin{equation}
D(fg)=\phi(f)D(g)+\phi(g)D(f).
\end{equation}
The set of all $\phi$-derivations with values in $M$ has a natural $B$-module structure, and we will denote it by $\Der_K(\phi,M)$.
\end{definition*}

$\phi$-derivations behave well when composed with $B$-linear maps. More specifically, any $B$-linear map composed with a $\phi$-derivation is again a $\phi$-derivation.

\begin{prop}
Let $M,M'$ be two $B$-modules and let $G:M\to M'$ be a $B$-module homomorphism. If $D:A\to M$ is a $\phi$-derivation, then $G\circ D:A\to M'$ is also a $\phi$-derivation.
\end{prop}

\begin{proof}
Simply, for any $f,g\in A$
\begin{equation}
G(D(fg))=G(\phi(f)D(g)+\phi(g)D(f)) = \phi(f)G(D(g))+\phi(g)G(D(f)).
\end{equation}
\end{proof}

\noindent As a result, with every $B$-module homomorphism $G:M\to M'$ we can associate the $B$-linear map $\Der_K(\phi,G):\Der_K(\phi,M)\to\Der_K(\phi,M')$ defined for any $D\in\Der_K(\phi,M)$ by
\begin{equation}
\Der_K(\phi,G)(D)=G \circ D.
\end{equation}
This assignment gives rise to the covariant endofunctor $\Der_K(\phi,\cdot)$ in the category of $B$-modules.

\begin{prop}
If $M,M',M''$ are $B$-modules and $G:M\to M'$, $G':M'\to M''$ are $B$-module homomorphisms, then
\begin{equation}
\Der_K(\phi,G')\circ\Der_K(\phi,G) =\Der_K(\phi,G' \circ G) 
\end{equation}
and $\Der_K(\phi,\Id_M) = \Id_{\Der_K(\phi,M)}$.
\end{prop}

\begin{proof}
For any $D\in\Der(\phi,M)$ we have that
\begin{equation}
\begin{split}
\Der_K(\phi,G' \circ G) (D) &= (G'\circ G)\circ D = G' \circ (G \circ D)\\
&=\Der_K(\phi,G')(G\circ D)\\
&=\Der_K(\phi,G')( \Der_K(\phi,G)(D)) \\
&=(\Der_K(\phi,G')\circ \Der_K(\phi,G))(D).
\end{split}
\end{equation}
The second part follows directly from the definition of $\Der_K(\phi,\Id_M).$
\end{proof}

We are interested in finding a representation for the functor $\Der_K(\phi,\cdot)$. Such representations will be called universal $\phi$-derivations, and they will be the main object of our study in this section.

\begin{definition*}
A pair $(\Omega_\phi,d_\phi)$ consisting of a $B$-module $\Omega_\phi$ and a $\phi$-derivation $d_\phi:A\to \Omega_\phi$ will be called a \emph{universal $\phi$-derivation} if for any $B$-module $M$ the assignment
\begin{equation}
\label{eq_alpha}
\Hom_B(\Omega_\phi, M)\ni F \overset{\alpha_M}{\longmapsto} F\circ d_\phi \in \Der_K(\phi,M)
\end{equation} is an isomorphism of $B$-modules. Such $d_\phi$ will be called a \emph{$\phi$-differential} and such $\Omega_\phi$ will be called a \emph{module of $\phi$-differentials}.
\end{definition*}

In order to explain the word universal in the above definition, we just need to read what the bijectivity of $\alpha_M$ means. By doing this, we get the so called universal factorisation property, and as a result we obtain an equivalent description of universal $\phi$-derivations.

\begin{prop}
$(\Omega_\phi,d_\phi)$ is a universal $\phi$-derivation if and only if for any $B$-module $M$ and any $D\in\Der_K(\phi,M)$ there is a unique $B$-linear map $F:\Omega_\phi\to M$ such that $F\circ d_\phi=D.$
\end{prop}

In the special case of $B=A$ and $\phi=\Id_A$, we have that $(\Omega_{\Id_A},d_{\Id_A})$ is precisely the Kähler differential of $A$, which is usually denoted by $(\Omega_{A/K},d_{A/K})$. The relation between Kähler differentials and $\phi$-differentials will be investigated more deeply in Appendix \ref{app_Kähler}.

There is also another reason why we named $d_\phi$ a $\phi$-differential. In Section \ref{sec_Examples} we will show that some differentials appearing in algebra and geometry are in fact $\phi$-differentials for some $\phi$.

We left constructions of universal $\phi$-derivations for the next section because we want to show which results can be obtained by using only the universal factorisation property. Besides, introducing models of $(\Omega_\phi,d_\phi)$ at this point would only lead to unnecessary confusion. 

By using a standard argument applied to objects defined by a universal factorisation property, we obtain that there is a unique universal $\phi$-derivation up to a $B$-module isomorphism.

\begin{prop}
\label{prop_Der_UniqDer}
If $(\Omega_\phi,d_\phi)$ and $(\Omega'_\phi,d'_\phi)$ are two universal $\phi$-derivations, then there is a unique $B$-module isomorphism $F:\Omega_\phi\to\Omega'_\phi$ such that $F\circ d_\phi=d'_\phi$.
\end{prop}

\begin{proof}
From the universality of $(\Omega_\phi,d_\phi)$ there is a $B$-module homomorphism $F:\Omega_\phi\to\Omega'_\phi$ such that $F\circ d_\phi=d'_\phi$. Similarly, from the universality of $(\Omega'_\phi,d'_\phi)$ there is a $B$-module homomorphism $F':\Omega'_\phi\to\Omega_\phi$ such that $F'\circ d'_\phi=d_\phi$. Hence
\begin{equation}
(F'\circ F) \circ d_\phi = F' \circ (F \circ d_\phi) = F' \circ d'_\phi = d_\phi,
\end{equation}
and obviously $\Id_{\Omega_\phi}\circ d_\phi = d_\phi$. Thus, from the uniqueness of the factor of $d_\phi$, we get that $F'\circ F = \Id_{\Omega_\phi}$. Likewise, using an analogous argument, we get that $F\circ F' = \Id_{\Omega'_\phi}$, and so $F$ is an isomorphism.
\end{proof}

Due to Proposition \ref{prop_Der_UniqDer}, we will call such $(\Omega_\phi,d_\phi)$ \emph{the} universal derivation, and we mean by that a fixed but arbitrary universal $\phi$-derivation. Similarly, we will use the notions of \emph{the} $\phi$-differential and \emph{the} module of $\phi$-differentials for $d_\phi$ and $\Omega_\phi$ respectively. Moreover, for the sake of convenience, we will sometimes identify $d_\phi$ with $(\Omega_\phi,d_\phi)$.

So far $A,B$ and $\phi:A\to B$ were fixed. We leave $B$ untouched and allow $A$ and $\phi$ to vary in a controlled manner. Let us consider the comma category $(\Alg_K\downarrow B)$, where $\Alg_K$ stands for the category of $K$-algebras. From the definition of $(\Alg_K\downarrow B)$, objects in this category are $K$-algebra homomorphisms $\phi:A\to B$ with the fixed codomain $B$, and morphisms from $\phi:A\to B$ to $\phi':A'\to B$ are $K$-algebra homomorphisms $H:A\to A'$ such that $\phi'\circ H=\phi$. 

Objects of $(\Alg_K\downarrow B)$ can be visualised as roughly vertical arrows targeting $B$, and morphisms as horizontal arrows closing diagrams made of such objects. 

\begin{equation}
\xymatrix{
A\ar[rr]^{H}\ar[dr]_{\phi}&& A'\ar[dl]^{\phi'} \\
&B&
}
\end{equation}

By assigning modules of $\phi$-differentials $\Omega_\phi$ to $K$-algebra homomorphisms $\phi:A\to B$, we obtain the following mapping between objects in categories:
\begin{equation}
(\Alg_K\downarrow B)\ni (\phi: A\to B) \overset{\Omega}{\longmapsto} \Omega_\phi \in \Mod_B,
\end{equation}
where $\Mod_B$ stands for the category of $B$-modules. We will show that this assignment $\Omega$ is a functor. First, however, we have to make sure that there is a reasonable mapping from arrows of $(\Alg_K\downarrow B)$ to arrows of $\Mod_B$.

\begin{prop}
\label{prop_Der_ComaCat}
Let $\phi:A\to B,\phi':A'\to B$ and $H:A\to A'$ be $K$-algebra homomorphisms such that $\phi'\circ H=\phi$. Then there is a unique $B$-linear map $\Omega_H:\Omega_\phi\to\Omega_{\phi'}$ such that $\Omega_H\circ d_\phi = d_{\phi'}\circ H$.
\end{prop}

\begin{proof}
Consider the $K$-linear map $D:A \to \Omega_{\phi'}$ defined by the formula 
\begin{equation}
D(f) = d_{\phi'}(H(f)).
\end{equation}
$D$ is a $\phi$-derivation. In fact, for any $f,g\in A$
\begin{equation}
\begin{split}
D(fg) & = d_{\phi'}(H(fg))=d_{\phi'}(H(f)H(g))\\
&=\phi'(H(f))d_{\phi'}(H(g))+\phi'(H(g))d_{\phi'}(H(f))\\
& = \phi(f)D(g) + \phi(g)D(f).
\end{split}
\end{equation}
Thus there is a unique $B$-linear map $\Omega_H:\Omega_\phi\to\Omega_{\phi'}$ such that $\Omega_H\circ d_\phi = D$, and since $D = d_{\phi'}\circ H$, we get that $\Omega_H\circ d_\phi = d_{\phi'}\circ H$.
\end{proof}

Henceforth, $\Omega_H$ will denote the map introduced in Proposition \ref{prop_Der_ComaCat}. We can visualise $\Omega_H$ as the unique $B$-linear map which makes the diagram
\begin{equation}
\label{diagram}
\xymatrix{
\Omega_\phi\ar[rr]^{\Omega_H}&& \Omega_{\phi'}\\
A\ar[u]_{d_{\phi}}\ar[rr]^{H}\ar[dr]_{\phi}&& A'\ar[u]^{d_{\phi'}}\ar[dl]^{\phi'} \\
&B&
}
\end{equation}  
commutative. As a consequence, we obtain a well defined assignment between arrows of categories
\begin{equation}
(\Alg_K\downarrow B)\ni H \overset{\Omega}{\longmapsto} \Omega_H \in \Mod_B.
\end{equation}
The next proposition asserts us that this mapping is in fact a covariant functor.

\begin{prop}
If $H:A\to A'$ and $H':A'\to A''$ are morphisms of $(\Alg_K\downarrow B)$ from $\phi:A\to B$ to $\phi:A'\to B$ and from $\phi':A'\to B$ to $\phi'':A''\to B$ respectively, then $\Omega_{H'}\circ \Omega_H = \Omega_{H'\circ H}$ and $\Omega_{\Id_A} = \Id_{\Omega_\phi}$.
\end{prop}

\begin{proof}
\begin{equation}
\begin{split}
\Omega_{H'\circ H}\circ d_\phi &= d_{\phi''}\circ (H'\circ H)=(d_{\phi''}\circ H')\circ H = (\Omega_{H'}\circ d_{\phi'})\circ H\\
&=\Omega_{H'}\circ (d_{\phi'}\circ H)=\Omega_{H'}\circ(\Omega_H\circ d_\phi)=(\Omega_{H'}\circ \Omega_H)\circ d_\phi.
\end{split}
\end{equation}
From Proposition \ref{prop_Der_ComaCat} we have that there are unique maps closing diagrams of the type \eqref{diagram}. Thus
\begin{equation}
\Omega_{H'\circ H}\circ d_\phi = (\Omega_{H'}\circ \Omega_H)\circ d_\phi
\end{equation}
implies that $\Omega_{H'\circ H}=\Omega_{H'}\circ \Omega_H$. The second part of the proposition holds simply because $\Id_{\Omega_\phi}\circ d_\phi = d_\phi$.
\end{proof}

Now we consider a technical lemma which will help us in Section 2 with verifying whether some $\phi$-derivation is a universal one. The idea is to transform the uniqueness from the factorisation property into some sort of surjectivity of a $\phi$-derivation.

\begin{prop}
\label{prop_Der_UniFactor}
Let $\Omega$ be a $B$-module and let $d\in\Der_K(\phi,\Omega)$. Assume that for any $B$-module $M$ and any $D\in\Der_K(\phi,M)$ there exists a $B$-linear map $F:\Omega\to M$ such that $F\circ d= D$. Then the pair $(\Omega,d)$ is a universal $\phi$-derivation if and only if $\Span_B(\Ima d)=\Omega$.
\end{prop}

\begin{proof}
First assume that $(\Omega,d)$ is a universal $\phi$-derivation. Let $\Omega_0$ denote the $B$-module $\Span_B(\Ima d)$ and let $\iota:\Omega_0\hookrightarrow\Omega$ be the natural inclusion. We put $d_0$ as the map $d$ with the codomain truncated to $\Omega_0$. Thus we have that $\iota\circ d_0 = d$. $d_0$ is clearly a $\phi$-derivation, hence there is a $B$-linear map $F:\Omega\to \Omega_0$ such that $F\circ d = d_0$. As a result
\begin{equation}
(\iota \circ F) \circ d = \iota \circ (F\circ d) = \iota \circ d_0 = d,
\end{equation}
and clearly $\Id_{\Omega}\circ d = d$. From the uniqueness of the factor of $d$ we obtain that $\iota \circ F = \Id_{\Omega}$. Hence $\iota$ is surjective, and so $\Omega = \Omega_0$.

Now assume that $\Omega = \Omega_0$. Let $M$ be a $B$-module and let $D\in\Der_B(\phi,M)$. So there exist $F_1,F_2:\Omega\to M$ such that $F_1\circ d=D$ and $F_2\circ d = D$. We need to show that $F_1=F_2$. If $\omega = \sum b_i d(f_i)\in \Omega$, then
\begin{equation}
\begin{split}
F_1(\omega)&=F_1\left(\sum b_id(f_i)\right)=\sum b_i F_1(d(f_i))=\sum b_i F_2(d(f_i))\\
&=F_2\left(\sum b_id(f_i)\right)=F_2(\omega).
\end{split}
\end{equation}
The assumption $\Omega = \Omega_0$ implies that any $\omega\in \Omega$ can be written as $\sum b_i d(f_i)$, and so $F_1=F_2$.
\end{proof}

As a direct consequence, we obtain that the module of $\phi$-differentials is generated as a $B$-module by the image of the $\phi$-differential. This also implies that the functor $\Omega$ preserves surjectivity.

\begin{prop}
\label{prop_Der_SurPre}
Assume that $H:A\to A'$ is an arrow of $(\Alg_K\downarrow B)$ from $\phi:A\to B$ to $\phi:A'\to B$. If $H$ is surjective, then $\Omega_H:\Omega_\phi\to\Omega_{\phi'}$ is surjective as well.
\end{prop}

\begin{proof}
If $\alpha'\in \Omega_{\phi'}$, then, from Proposition \ref{prop_Der_UniFactor}, we get that $\alpha'=\sum b_id_{\phi'} (f'_i)$ for some elements $b_i\in B$ and $f'_i\in A'$. The surjectivity of $H$ implies that $f'_i=H(f_i)$ for some $f_i\in A$, and thus $d_{\phi'}f'_i=\Omega_H(d_\phi (f_i))$. As a result,
\begin{equation}
\alpha'=\sum b_id_{\phi'}(f'_i)=\sum b_i\Omega_H(d_\phi (f_i))=\Omega_H\left(\sum b_id_\phi (f_i)\right).
\end{equation}
\end{proof}

We will make use of Proposition \ref{prop_Der_SurPre} in Section \ref{sec_AffVar} when we investigate linear differential operators acting on the ring of regular functions of an affine variety. In a nutshell, this will be useful because for any affine variety there is a natural surjection from the algebra of polynomials to the algebra of regular functions.

In the realm of smooth manifolds, the differential at a point takes smooth functions and outputs cotangent vectors at this point. By definition, cotangent vectors are the dual objects to tangent vectors, which themselves may be seen as $\phi$-derivations for $\phi$ being an evaluation at a point. Due to this behaviour of the ordinary differential, we will now consider this kind of biduality in a more abstract framework. If $M$ is a $B$-module, then the dual $B$-module will be denoted by $M^\vee$. This means that $M^\vee = \Hom_B(M,B)$. Let $d^M_\phi:A\to\Der_K(\phi,M)^\vee$ be the map defined by the formula
\begin{equation}
\label{eq_AlgDiff}
d^M_\phi(f)(D) = D(f),
\end{equation}
where $f\in A$ and $D\in\Der_K(\phi,M)$.

\begin{prop}
For any $B$-module $M$ the map $d^M_\phi$ is a $\phi$-derivation.
\end{prop}

\begin{proof}
Simply, for any $f,g\in A$ and $D\in\Der_K(\phi,M)$ we have that
\begin{equation}
\begin{split}
d^M_\phi(fg)(D)&=D(fg)=\phi(f)D(g)+\phi(g)D(f) \\
&=(\phi(f)d^M_\phi(g)+\phi(g)d^M_\phi(f))(D).
\end{split}
\end{equation}
The $K$-linearity of $d^M_\phi$ is clear because everything here is $K$-linear.
\end{proof}

\noindent Consider now the special case where $M=B$. In this situation we have a necessary and sufficient condition for $d_\phi^B:A\to\Der_K(A,B)^\vee$ to be a $\phi$-differential.

\begin{prop}
\label{prop_Der_BidualDiff}
$(\Der_K(\phi,B)^\vee, d^B_\phi)$ is a universal $\phi$-derivation if and only if the module of $\phi$-differentials $\Omega_\phi$ is reflexive as a $B$-module.
\end{prop}

\begin{proof}
Let $j:\Omega_\phi\to\Omega_\phi^{\vee\vee}$ be the natural map to the bidual. The reflexivity of $\Omega_\phi$ means that $j$ is an isomorphism of $B$-modules. The universality of $\Omega_\phi$ implies that $B$-modules $\Omega_\phi^\vee$ and $\Der_K(\phi,B)$ are isomorphic via the map $\alpha_B$ defined as in \eqref{eq_alpha}. Let $\beta$ denote the inverse map of $\alpha_B$. This isomorphism $\beta:\Der_K(\phi,B)\to\Omega_\phi^\vee$, from the definition of $\alpha_B$, is such that for any $f\in A$ and any $D\in\Der_K(\phi,B)$
\begin{equation}
\beta(D)(d_\phi(f)) = D(f).
\end{equation}
We have that $(\beta^\vee\circ j)\circ d_\phi=d^B_\phi,$ where $\beta^\vee:\Omega_\phi^{\vee\vee}\to\Der_K(\phi,B)^\vee$ is the dual map to $\beta.$ In fact, for any $f\in A$ and $D\in\Der_K(\phi,B)$
\begin{equation}
(\beta^\vee(j(d_\phi(f)))(D)=(j(d_\phi(f))(\beta(D))=\beta(D)(d_\phi(f))=D(f)=(d^B_\phi(f))(D).
\end{equation}
Thus $(\Der_K(\phi,B)^\vee, d^B_\phi)$ is a universal $\phi$-derivation if and only if $\beta^\vee\circ j$ is an isomorphism. But since $\beta$ is an isomorphism, we get that $\beta^\vee\circ j$ is an isomorphism if and only if $j$ is an isomorphism.
\end{proof}

The case which will concern us the most is the one in which $B$ is a field. If a vector space is finite dimensional, then it is reflexive. Thus, as a corollary, we get a degenerated version of Proposition \ref{prop_Der_BidualDiff}.

\begin{corollary}
\label{cor_Der_BidualDiff}
If $B$ is a field and the module of $\phi$-differentials $\Omega_\phi$ is a finite dimensional vector space over the field $B$, then $d_\phi^B$ is a $\phi$-differential.
\end{corollary}

\section{Three models of the universal $\phi$-derivation}
\label{sec_Models}

We have from Proposition \ref{prop_Der_UniqDer} that all universal $\phi$-derivations are isomorphic and from perspective of $B$-modules they are the same object. However, if we wish to find a $K$-linear map from $\Omega_{\phi}$ to some $K$-module, then we cannot rely only on the universal property of $(\Omega_\phi,d_\phi)$. In such case, it may happen that we have to refer to models of the universal $\phi$-derivation. Such issue appears with Kähler differentials in the construction of the algebraic de-Rham complex. Compare for instance constructions of the algebraic de-Rham complex provided by \cite[thm 16.6.2]{Grothendieck} and by \cite[prop. 4.5.3]{Brandenburg}. The latter is much more straightforward due to the use of a different model of the Kähler differential. 

In this section, we will give detailed constructions of three notably different models of the universal $\phi$-derivation. For the sake of unambiguity and to omit confusions we will give a distinct name to each model.

\bigskip
\noindent\textbf{First model: Classical universal $\phi$-derivation}
\medskip

Let us consider $B\otimes_K A$ with the following $B$-module structure:
\begin{equation}
a \cdot b\otimes f = ab\otimes f,
\end{equation}
and with the following $B$-algebra structure:
\begin{equation}
a\otimes f \cdot b\otimes g = ab\otimes fg.
\end{equation}
Let $p_\phi:B\otimes_K A\to B$ be a $B$-algebra homomorphism given by the formula
\begin{equation}
\label{eq_pphi}
p_\phi(a\otimes f) = a\phi(f),
\end{equation} 
and let $\delta_\phi:A\to B\otimes_K A$ be a $K$-linear map defined as
\begin{equation}
\delta_\phi(f)= 1_B\otimes f - \phi(f)\otimes 1_A.
\end{equation}
Let $I_\phi = \ker p_\phi$ and let $J_\phi$ be the $B$-module generated by $\Ima\delta_\phi$. I.e., $J_\phi=\Span_B(\Ima\delta_\phi).$ We claim that $I_\phi=J_\phi$. In fact, the inclusion $I_\phi\supset J_\phi$ follows from the fact that $p_\phi\circ \delta_\phi=0$. Conversely, consider an arbitrary $\sum a_i\otimes f_i\in I_\phi$. Then
\begin{equation}
\sum a_i\otimes f_i = \sum (a_i\otimes 1_A)(1_B\otimes f_i - \phi(f_i)\otimes 1_A) = \sum a_i\delta_\phi (f_i).
\end{equation} 
This shows that $I_\phi\subset J_\phi$. $I_\phi\subset B\otimes_K A$ is an ideal because it is the kernel of a $B$-algebra homomorphism. Hence $I^2_\phi$ is a well defined $B$-module. We define the $B$-module $\Omega^C_{\phi} = I_\phi/I^2_\phi$ and the $K$-linear map $d^C_{\phi}:A\to\Omega_{\phi}$ as $\pi^C\circ \delta_\phi$, where $\pi^C:I_\phi\to \Omega_{\phi}^C$ is the natural projection.

We are aiming to prove that $d^C_{\phi}$ is a $\phi$-differential. First, however, we need two lemmas. The first one states that $\phi$-derivations vanish on the identity.

\begin{lemma}
\label{lem_Mod_DerOn1}
For any $B$-module $M$ and any $D\in \Der_K(\phi,M)$ we have that $D(1_A)=0$.
\end{lemma}

\begin{proof}
Simply,
\begin{equation}
D(1_A)=D(1_A\cdot 1_A)=\phi(1_A)D(1_A)+\phi(1_A)D(1_A)=2\phi(1_A)D(1_A).
\end{equation}
Since $\phi$ is a $K$-algebra homomorphism, we get that $\phi(1_A)=1_B$, and so $D(1_A)=0$.
\end{proof}

\noindent The second lemma tells us that $\delta_\phi$ behaves almost like a $\phi$-derivation.

\begin{lemma}
\label{lem_Mod_Delta}
For any $f,g\in A$
\begin{equation}
\label{eq_delta}
\delta_\phi(fg)-\phi(f)\delta_\phi(g)-\phi(g)\delta_\phi(f) = \delta_\phi(f)\delta_\phi(g).
\end{equation}
\end{lemma}

\begin{proof}
First, from the definition of $\delta_\phi$, we have that
\begin{equation}
\begin{split}
\delta_\phi(f)\delta_\phi(g)&=(1\otimes f - \phi(f)\otimes 1)(1\otimes g - \phi(g)\otimes 1)\\
&=1\otimes fg - \phi(f)\otimes g - \phi(g)\otimes f +\phi(fg)\otimes 1.
\end{split}
\end{equation}
By adding and subtracting $\phi(fg)\otimes 1$ we obtain that the above equals
\begin{equation}
(1\otimes fg - \phi(fg)\otimes 1) - (\phi(f)\otimes g - \phi(fg)\otimes 1) - (\phi(g)\otimes f - \phi(fg)\otimes 1),
\end{equation}
and this is
\begin{equation}
\delta_\phi(fg)-\phi(f)\delta_\phi(g)-\phi(g)\delta_\phi(f).
\end{equation}
\end{proof}

\begin{theorem}
$(\Omega^C_\phi,d^C_\phi)$ is a universal $\phi$-derivation.
\end{theorem}

\begin{proof}
We have from Lemma \ref{lem_Mod_Delta} that $d_{\phi}^C$ is a $\phi$-derivation. The module $\Omega^C_{\phi}$ equals $\Span_B(\Ima d^C_{\phi})$ because $I_\phi=J_\phi=\Span_B(\Ima \delta_\phi)$ and $d^C_{\phi}=\pi^C\circ \delta_\phi$. Thus, from Proposition \ref{prop_Der_UniFactor}, we just need to find a factor of an arbitrary $\phi$-derivation. So let $M$ be any $B$-module and let $D\in\Der_K(\phi,M)$. Consider the $B$-linear map $\gamma: B\otimes_K A \to M$ given by the formula
\begin{equation}
\gamma(a\otimes f) = aD(f) 
\end{equation}
and truncate it to $\bar{\gamma}:I_\phi \to M$. The map $\bar{\gamma}$ vanishes on $I_\phi^2.$ In fact, using Lemma \ref{lem_Mod_DerOn1}, we first notice that for any $f\in A$
\begin{equation}
\bar{\gamma}(\delta_\phi(f)) = D(f)-\phi(f)D(1_A) = D(f).
\end{equation}
Hence, using \eqref{eq_delta} from Lemma \ref{lem_Mod_Delta}, we get that for any $f,g\in A$
\begin{equation}
\begin{split}
\bar{\gamma}(\delta_\phi(f)\delta_\phi(g))&= \bar{\gamma}(\delta_\phi(fg)-\phi(f)\delta_\phi(g)-\phi(g)\delta_\phi(f))\\
&=D(fg)-\phi(f)D(g)-\phi(g)D(f)=0.
\end{split}
\end{equation}
Thus there is a $B$-linear map $F:\Omega_{\phi}^C\to M$ such that $F\circ d_{\phi}^C = D$.
\end{proof}

\begin{definition*}
We will call $(\Omega^C_\phi,d^C_\phi)$ the \emph{classical universal $\phi$-derivation}.
\end{definition*}

We call it classical because it is just a modification of the standard construction of the Kähler differential appearing in classic sources like \cite[p. 568-569]{Bourbaki} or \cite[p. 210-215]{H-S}. This model will be most interesting for us because, as we will see in Section \ref{sec_Examples}, modules $\Omega^C_\phi$ for particular homomorphisms $\phi$ appear naturally in algebra and geometry. Moreover, in Sections \ref{sec_Elliptic} and \ref{sec_AffVar} we will extensively use the map $\delta_\phi$ introduced in this model.

\bigskip
\noindent\textbf{Second model: Rough universal $\phi$-derivation}
\medskip

Let $B^{(A)}$ be the free $B$-module generated by the set $A$. Consider its $B$-submodule $S$ generated by the set
\begin{equation}
\left\lbrace \left\langle k\right\rangle :k\in K\right\rbrace\cup \left\lbrace \left\langle fg \right\rangle -\phi(f)\left\langle g \right\rangle-\phi(g)\left\langle f \right\rangle:f,g\in A\right\rbrace,
\end{equation} 
where $\left\langle f\right\rangle:A\to B$ stands for the base element of $B^{(A)}$ corresponding to $f$. This means that
\begin{equation}
\left\langle f \right\rangle(g) = 
\begin{cases}
1_B,& \text{if } g=f\\
0,  & \text{if } g\neq f
\end{cases}.
\end{equation}
We define the $B$-module $\Omega_{\phi}^R$ as the quotient $B^{(A)}/S$ and the $K$-linear map $d^R_{\phi}:A\to\Omega^R_{\phi}$ by the formula
\begin{equation}
d^R_{\phi}(f)=\pi^R(\left\langle f \right\rangle),
\end{equation} 
where $\pi^R:B^{(A)}\to \Omega^R_{\phi}$ is the natural projection.

\begin{theorem}
$(\Omega^R_\phi,d^R_\phi)$ is a universal $\phi$-derivation.
\end{theorem}

\begin{proof}
Let $M$ be a $B$-module and let $D\in\Der_K(\phi,M)$. Consider the $B$-linear map $\gamma:B^{(A)}\to M$ defined on base elements by the formula
\begin{equation}
\gamma(\left\langle f \right\rangle)=D(f).
\end{equation}
Since $D$ is a $\phi$-derivation, it follows that $\gamma$ vanishes on $S$. Hence there is a unique $B$-linear map $F:\Omega^R_{\phi}\to M$ such that $D=F\circ  d^R_{\phi}$.
\end{proof}

\begin{definition*}
We will call $(\Omega^R_\phi,d^R_\phi)$ the \emph{rough universal $\phi$-derivation}.
\end{definition*}

The above construction is very similar to the construction of the tensor product. We consider a huge free module and divide it by a submodule generated by appropriately tailored elements. This rough approach explains our choice of the name. This model is also a modification of some model of the Kähler differential (see \cite[p. 384]{Eisenbud}). 

\bigskip
\noindent\textbf{Third model: Algebraic universal $\phi$-derivation}
\medskip

Let us consider $B\otimes_K A\otimes_K A$ and $B\otimes_K A$ with the natural $B$-module structures endowed by the left most factors. Let $h_\phi:B\otimes_K A\otimes_K A\to B\otimes_K A$ be a $B$-linear map defined by the formula
\begin{equation}
h_\phi(a\otimes f\otimes g) = a\otimes fg - a\phi(f)\otimes g - a\phi(g)\otimes f.
\end{equation}
We define the $B$-module $\Omega^A_{\phi}$ as the cokernel of $h_\phi$. This means that $\Omega^A_{\phi}=(B\otimes_K A)/ \Ima(h_\phi)$, and we let $d^A_{\phi}:A\to \Omega^A_{\phi}$ be defined by the formula
\begin{equation}
d^A_{\phi}(f)=\pi^A(1\otimes f),
\end{equation}
where $\pi^A:B\otimes_K A\to \Omega^A_{\phi}$ is the natural projection.

\begin{theorem}
$(\Omega^A_\phi,d^A_\phi)$ is a universal $\phi$-derivation.
\end{theorem}

\begin{proof}
Let $M$ be a $B$-module and let $D\in\Der_K(\phi,M)$. Consider the $B$-linear map $\gamma:B\otimes_K A\to M$ defined by the formula
\begin{equation}
\gamma(a\otimes f)=aD(f).
\end{equation}
Then for any $a\in B$ and $f,g\in A$ we have that
\begin{equation}
\gamma(h_\phi(a\otimes f\otimes g))=aD(fg)-a\phi(f)D(g)-a\phi(g)D(f),
\end{equation}
and since $D$ is a $\phi$-derivation, we get that $\gamma\circ h_\phi = 0$. Thus there is a unique $B$-linear map $F:\Omega^A_{\phi}\to M$ such that $D=F\circ  d^A_{\phi}$.
\end{proof}

\begin{definition*}
We will call $(\Omega^A_\phi,d^A_\phi)$ the \emph{algebraic universal $\phi$-derivation}.
\end{definition*}

The above construction is just a modification of the model of the Kähler differential discovered in \cite[p. 74-79]{Brandenburg}.

\section{Examples of $\phi$-differentials}
\label{sec_Examples}

The aim of this section is to show that some differentials appearing in algebra and geometry are in fact $\phi$-differentials.

First of all, as we have seen earlier, the notion of $\phi$-differential is a straightforward generalisation of the notion of Kähler differential. More precisely, for a given $K$-algebra $A$ the Kähler differential of $A$ is the same as the $\Id_A$-differential. However, it is merely a consequence of definitions we introduced, rather then an actual example.

The case which is the most interesting for us is the one in which, using notations from previous sections, the $K$-algebra $B$ is just $K$ itself. This is due to the philosophy from \cite[ch. 9]{Nestruev}, which suggests to look at differential calculus as a part of commutative algebra. So here is a brief overview of the main concepts which will concern us the most. For more details we refer to \cite[ch. 9]{Nestruev}.

Given a $K$-algebra $A$, we denote the set of all $K$-algebra homomorphisms $h:A\to K$ by $|A|$. $|A|$ is called the \emph{$K$-spectrum of $A$} and its elements are called \emph{$K$-points}. For any $h\in|A|$ the \emph{cotangent space at $h$} is defined as $I_h/I_h^2$ and it is denoted by $T^*_h A$, where $I_h=\ker h$. Additionally, we allow ourselves to define the \emph{differential at $h$} as the map $d_h:A\to T^*_h A$ given by the formula $d_h(f)=[f-h(f)]$.

Since any $K$-point $h$ is a $K$-algebra homomorphism, it is reasonable to talk about $h$-derivations. Hence let us focus for a moment on how the classical universal $h$-derivation looks like. Since $h$ takes values in $K$, we have that $\Omega^C_h$ is precisely the cotangent space $T_h^* A$, and similarly the $h$-differential $d^C_h$ is just $d_h$. As a result, we obtain that $d_h$ is a $h$-differential.

\begin{prop}
\label{prop_Ex_Alg}
For any $K$-algebra $A$ and any $h\in|A|$ the pair $(T_h^* A, d_h)$ is a universal $h$-derivation.
\end{prop}

The above algebraic framework is general enough to work with various cases in geometry. We will show that differentials appearing in both algebraic and differential geometry are $\phi$-differentials for $\phi$ being evaluations at a point. 

Let $\mathbb{F}$ be a field and assume that $V$ is an affine variety given by an ideal $I\subset \mathbb{F}[X_1,\dots,X_n]$. Let $\mathbb{F}[V]=\mathbb{F}[X_1,\dots,X_n]/I$ be the associated ring of regular functions on $V$. For a given point $z\in V$ let $\ev_z:\mathbb{F}[V]\to\mathbb{F}$ denote the evaluation function at $z$. By so, we mean that $\ev_z(\overline{f})=f(z)$, where $f\in\mathbb{F}[X_1,\dots,X_n]$ is any representative of $\overline{f}\in\mathbb{F}[V]$. The \emph{Zariski cotangent space at $z$} is defined as $\mathfrak{m}_z/\mathfrak{m}_z^2$, where $\mathfrak{m}_z$ is the maximal ideal of regular functions vanishing at $z$ and it is denoted by $T_z^*V$. Equivalently, $T_z^*V=I_z/I_z^2$, where $I_z=\ker(\ev_z)$. Finally, the \emph{differential at $z$} is the map $d_z:\mathbb{F}[V]\to T^*_z V$ given by the formula $d_z(\overline{f})=[\overline{f}-\ev_z(\overline{f})]$. 

Since $\ev_z$ is clearly an $\mathbb{F}$-algebra homomorphism, we get that this geometric situation is a special case of the previous algebraic one with $K=\mathbb{F}$, $A=\mathbb{F}[V]$ and $h=\ev_z$. Hence the differential $d_z$ is an $\ev_z$-differential.

\begin{prop}
\label{prop_Ex_AG}
For any affine variety $V$ over a field $\mathbb{F}$ and any point $z\in V$ the pair $(T_z^* V, d_z)$ is a universal $\ev_z$-derivation.
\end{prop}

A similar result holds for ordinary differentials defined on smooth manifolds (by smooth we mean $C^\infty$). Let $X$ be a smooth real Hausdorff manifold and let $z\in X$ be some point. Assume that $\ev_z:C^\infty(M)\to\mathbb{R}$ is the standard evaluation map at $z$. \emph{Tangent vectors at $z$} are defined as $\ev_z$-derivations with values in $\mathbb{R}$ and the vector space of all tangent vectors at $z$ is denoted by $T_zX$. The \emph{cotangent space at $z$} is the dual space to $T_zX$ and it is denoted by $T^*_zX$. The \emph{differential at $z$} is the map $d_z:C^\infty(M)\to T^*_zX$ given by the formula $d_z(f)(v)=v(f)$. 

Straight from the definition of $T_zX$ we have that $T_zX = \Der_\mathbb{R}(\ev_z,\mathbb{R})$. As a result, the differential at $z$ is precisely the map $d^\mathbb{R}_{\ev_z}:C^\infty(X)\to\Der_\mathbb{R}(\ev_z,\mathbb{R})^{\vee}$ introduced in \eqref{eq_AlgDiff} with $A=C^\infty(X)$, $M=B=K=\mathbb{R}$ and $\phi=\ev_z$. So in order to prove that $d_z$ is an $\ev_z$-differential, from Corollary \ref{cor_Der_BidualDiff}, it is enough to prove that the module of $\ev_z$-differentials $\Omega_{\ev_z}$ is a finite dimensional real vector space. This clearly holds. In fact, from the universal property of the universal $\ev_z$-derivation we have that the dual vector space to $\Omega_{\ev_z}$ is precisely the tangent space $T_zX$, which has a finite dimension.

\begin{prop}
\label{prop_Ex_DG}
For any smooth real Hausdorff manifold $X$ and any point $z\in X$ the pair $(T_z^*X,d_z)$ is a universal $\ev_z$-derivation.
\end{prop}

Since universal $\phi$-derivations are unique up to an isomorphism commuting with differentials, we get from Section \ref{sec_Models} that $(T_z^*X,d_z)$ has alternative descriptions. From the construction of the classical universal $\ev_z$-derivation, we have that $T_z^*X$ can be seen as $I_z/I_z^2$, where $I_z=\ker(\ev_z)$ with $d_z$ mapping a smooth function $f$ to the the class $[f-f(z)]$ in $I_z/I_z^2$. As a result, which was noticed in \cite[sec. 9.27]{Nestruev}, this geometric case can be seen as a special case of the previous algebraic one with $K=\mathbb{R}$, $A=C^\infty(X)$ and $h=\ev_z$. 

From the construction of the algebraic universal $\ev_z$-derivation, we get that $T_z^*X$ is isomorphic with $C^\infty(M)/S_z$, where $S_z$ is a vector space generated by elements of the form $fg-f(z)g-g(z)f$. Interestingly, it coincides with the space of differentials, introduced in \cite[sec. 2.1.9]{Narasimhan}, which is defined as the space of smooth functions modulo stationary functions at the point $z$.

In summary, even though the notion of $h$-differential is defined via the universal property, it generalises differentials from distinct categories which have very concrete set theoretic models. Additionally, being defined axiomatically and model independent has some practical applications. For example, in order to determine whether some $K$-linear map from a $K$-algebra $A$ factors through some $h$-differential, it is enough to check if it is a $h$-derivation. We will make use of this observation in upcoming Section \ref{sec_Elliptic} to define the principal symbol at a $K$-point in a purely algebraic way.

\section{Ellipticity in differential calculus over commutative algebras}
\label{sec_Elliptic}

Let us go back to the algebraic framework from the first section. So let $A,B$ be two $K$-algebras over some fixed ring $K$. For any $A$-module $P$ and any $B$-module $M$ there is a natural associative action of the $K$-algebra $B\otimes_K A$ on $\Hom_K(P,M)$ given by the formula
\begin{equation}
((a\otimes f)F)(p)=aF(fp),
\end{equation}
where $a\in B$, $f\in A$, $F\in\Hom_K(P,M)$ and $p\in P$. From the construction of the classical universal $\phi$-derivation we have that for a given $K$-algebra homomorphism $\phi:A\to B$ there is the $K$-linear map $\delta_\phi:A\to B\otimes_K A$ defined by the formula
\begin{equation}
\delta_\phi(f)=1_B\otimes f - \phi(f)\otimes 1_A.
\end{equation}
So, in particular, for $F\in\Hom_K(P,M)$, $f\in A$ and $p\in P$ we have that
\begin{equation}
(\delta_\phi(f)F)(p) = F(fp) - \phi(f)F(p).
\end{equation}
If $A=B$ and $\phi=\Id_A$, then we will denote $\delta_\phi$ by $\delta$. 

Using $\delta$, just like it is done in \cite[sec. 9.67]{Nestruev}, we can define a linear differential operator between two modules over a commutative algebra.

\begin{definition*}
Let $A$ be an arbitrary $K$-algebra and let $P,Q$ be two $A$-modules. A $K$-linear map $L:P\to Q$ is called a \emph{linear differential operator (LDO) of order $\leqslant k$} if for any $f_0,\dots,f_k\in A$
\begin{equation}
\delta (f_0)\dots\delta (f_k) L = 0.
\end{equation}
Naturally, if $L$ is of order $\leqslant k$ but not of order $\leqslant k-1$, then we will say that \emph{$L$ is of order $k$}.
\end{definition*}

This algebraic definition coincides with the usual definition of linear differential operators acting between global sections of smooth vector bundles on a smooth manifold \cite[sec. 9.66]{Nestruev}. Similarly, in the case of the algebra of polynomials $\mathbb{F}[X_1,\dots,X_n]$ over a field $\mathbb{F}$ of characteristics $0$, we have from \cite[sec. 16.11]{Grothendieck} that LDOs in the usual sense acting on $\mathbb{F}[X_1,\dots,X_n]$ coincide with the above algebraic ones with $P=Q=\mathbb{F}[X_1,\dots,X_n]$. Hence the above algebraic notion of LDO is coherent with the standard notion of LDO, and so it does not lead to any ambiguity.

Let us go back to the smooth case. If $X$ is a smooth manifold and $L:\Gamma(E)\to\Gamma(F)$ is an LDO of order $k$ on $X$ acting between smooth sections of some vector bundles $E$ and $F$, then we can use the principal symbol to study the character of $L$. The \emph{principal symbol at $z\in X$ of $L$} may be defined intrinsically as the function $\sigma_{L,z}:T^*_z X\otimes_\mathbb{R} E_z\to F_z$ given by the formula
\begin{equation}
\sigma_{L,z}(\alpha\otimes \xi) = L_z(f^k s),
\end{equation}
where $f\in C^\infty(M)$ and $s\in\Gamma(E)$ are such that $d_zf=\alpha$, $f(z)=0$ and $s(z)=\xi$ \cite[def. 3.3.13]{Narasimhan}. If for any non-zero $\alpha\in T^*_zM$ the $\mathbb{R}$-linear map $\sigma_{L,z}(\alpha):E_z\to F_z$ is injective, then we say that \emph{$L$ is elliptic at $z$}.

We wish to generalise the notions of principal symbol and ellipticity to the case of linear differential operators acting between modules over an arbitrary commutative $K$-algebra $A$.

We know from \cite[sec. 1.5.3]{KrLyVi} that with every LDO $L:P\to Q$ of order $\leqslant k$ we can associate an $A$-linear map $\sigma(L):(\Omega_{A/K})^{\odot k}\otimes_A P\to Q$ such that
\begin{equation}
\sigma(L)((d_{A/K}(f_1)\odot\dots\odot d_{A/K}(f_k))\otimes p) = \left(\delta(f_1)\dots\delta(f_k)L\right)(p),
\end{equation}
where $f_1,\dots,f_k\in A$, $p\in P$, and where $(\Omega_{A/K})^{\odot k}$ denotes the $k$-th symmetric power of the $A$-module $\Omega_{A/K}$. When $A=C^\infty(M)$ and $P,Q$ are modules of sections of some vector bundles, then this map $\sigma(L)$ corresponds to the global section of the principal symbol treated as a bundle map. Since the functor of global section $\Gamma$ behaves neatly on manifolds, in this geometric case it is quite easy to recover $\sigma_{L,z}$ for any $z$ from $\sigma(L)$ (just apply \cite[thm 1]{Swan}). However, in the case of an arbitrary algebra $A$ and $A$-modules $P,Q$ a more subtle approach is required.

First we notice that the principal symbol at a point is just a map between fibres of vector bundles. So in order to generalise this notion to our algebraic case, we need to have an algebraic analogue of points for an algebra and fibres for a module. 

We know from \cite[sec. 7.3]{Nestruev} that on a paracompact manifold $X$ ordinary points corresponds bijectively to $\mathbb{R}$-points of the algebra $C^\infty(X)$. More precisely, every $\mathbb{R}$-algebra homomorphism from $C^\infty(X)$ to $\mathbb{R}$ must be of the form $\ev_z$ for some $z\in X$. Hence given a $K$-algebra $A$ and an LDO acting between $A$-modules we are going to define the principal symbol at $h$ for any $K$-point $h\in|A|$. 

We know from \cite[cor. 3]{Swan} that for any smooth vector bundle $E$ on a manifold $X$ and any $z\in X$ the fibre $E_x$ is isomorphic with $\Gamma(E)/I_z\Gamma(E)$, where $I_z=\ker(\ev_z)$. As a result, repeating the ideas from \cite[cor. 11.9]{Nestruev}, we obtain a natural way to define an analogue of fibres at $K$-points for an arbitrary $A$-module. More specifically, for any $A$-module $P$ and any $K$-point $h\in |A|$ let $P_h$ be the quotient $K$-module $P/I_hP$, where $I_h$ is the kernel of $h$. We will denote the natural projection $P\to P_h$ by $\ev_h$, and for $p\in P$ we will denote the value $\ev_h(p)$ by $p_h$. Similarly, for any $K$-linear map $F:P\to Q$ we will denote the composition $\ev_h\circ F:P\to Q_h$ by $F_h$. In the special case of $P=A$ we have that $A_h=A/I_h$. As a result, $A_h$ can be canonically identified with $K$ via $h$. Thus, using this identification, the map $\ev_h:A\to A_h$ is just $h$ itself.

\begin{lemma}
\label{lem_El_Commute}
Let $P,Q$ be two $A$-modules, $h\in |A|$ and let $F:P\to Q$ be a $K$-linear map. Then for any natural number $j$ and any $f_1,\dots,f_j\in A$
\begin{equation}
\label{eq_induction}
\ev_h\circ\delta (f_1)\dots \delta (f_j) F = \delta_h (f_1)\dots \delta_h (f_j)F_h.
\end{equation}
\end{lemma}

\begin{proof}
We will show it inductively. First, for any $f\in A$ we have that
\begin{equation}
\label{eq_El_LemmaCommute}
\ev_h\circ (\delta (f) F) = \delta_h(f) F_h.
\end{equation}
In fact, for any $p\in P$ we have that
\begin{equation}
\begin{split}
\ev_h((\delta (f) F)(p)) &= \ev_h(F(fp)-fF(p))=F_h(fp)-h(f)F_h(p) \\
&= (\delta_h(f) F_h)(p).
\end{split}
\end{equation}
Now fix $f_0,\dots,f_j\in A$ and put $\Phi = \delta (f_1)\dots \delta (f_j) F$. Then we have that
\begin{equation}
\begin{split}
\ev_h \circ (\delta (f_0)\delta (f_1)\dots \delta (f_j) F)&=\ev_h\circ (\delta (f_0)\Phi)\\
\comment{\eqref{eq_El_LemmaCommute}\textnormal{ with }f=f_0\textnormal{ and } F=\Phi}&= \delta_h(f_0)\Phi_h\\
\comment{\ev_h\circ\Phi = \Phi_h}&= \delta_h(f_0)(\ev_h\circ\Phi)\\
\comment{\textnormal{inductive assumption }\eqref{eq_induction}}&= \delta_h(f_0)\delta_h(f_1)\dots\delta_h(f_j)F_h.
\end{split}
\end{equation}
\end{proof}

In the realm of smooth manifolds, given an $\mathbb{R}$-linear map between sections of vector bundles, we can ask whether this map gives rise to a linear map on fibres. Such mappings are precisely linear maps over the ring of smooth functions, which are usually called tensors. The following lemma is an algebraic analogue of that fact, but reduced to a single fibre.

\begin{lemma}
\label{lem_El_TensorLike}
Let $P$ be an $A$-module, $V$ be a $K$-module and let $h\in |A|$. If $G:P\to V$ is $K$-linear and for every $f\in A$ we have that $\delta_h (f) G=0$, then the map $G^h:P_h\to V$ given by the formula
\begin{equation}
G^h(p_h)=G(p)
\end{equation}
is well defined.
\end{lemma}

\begin{proof}
Let $p=f_0p_0$ for some $f_0\in I_h$ and $p_0\in P.$ Then
\begin{equation}
0=(\delta_h (f_0) G)(p_0) = G(f_0p_0) - h(f_0)G(p_0).
\end{equation}
Since $f_0\in I_h$, we obtain that $h(f_0)=0$. As a result, $G(f_0p_0)=0$.
\end{proof}

Now we are able to show the existence of the principal symbol at a $K$-point of an arbitrary linear differential operator.

\begin{theorem}
\label{thm_El_Symbol}
Let $P,Q$ be two $A$-modules and let $L:P\to Q$ be an LDO of order $\leqslant k$. Then for any $h\in |A|$ there exists the unique $K$-linear map
\begin{equation}
\sigma_{L,h}:\left(\Omega_h\right)^{\odot k}\otimes_K P_h\to Q_h
\end{equation}
such that for any $f\in I_h$ and any $p\in P$
\begin{equation}
\sigma_{L,h}((d_h(f))^{\odot k}\otimes p_h) = L_h(f^k p),
\end{equation}
where $\left(\Omega_h\right)^{\odot k}$ stands for the $k$-th symmetric power of the $K$-module $\Omega_h$.
\end{theorem}

\begin{proof}
Fix $f_1,\dots ,f_k \in A$ and consider the $K$-linear map $G_{f_1,\dots,f_k}:P\to Q_h$ given by the formula
\begin{equation}
G_{f_1,\dots,f_k} = \delta_h (f_1)\dots \delta_h (f_k) L_h.
\end{equation}
From the assumption that $L$ is an LDO of order $\leqslant k$ and Lemma \ref{lem_El_Commute} we obtain that for any $f\in A$ the function $\delta_h (f) G_{f_1,\dots,f_k}=0$. Thus, from Lemma \ref{lem_El_TensorLike} applied to $G_{f_1,\dots,f_k}$, we get that there exists a $K$-linear map $G^h_{f_1,\dots,f_k}:P_h\to Q_h$ such that for any $p\in P$
\begin{equation}
G^h_{f_1,\dots,f_k}(p_h)=\left(\delta_h (f_1)\dots \delta_h (f_k) L_h\right)(p).
\end{equation}
Hence, by changing $f_1,\dots,f_k$, we obtain the $K$-linear mapping
\begin{equation}
\label{eq_multiderivation}
A^{\otimes k} \ni f_1\otimes\dots\otimes f_k \mapsto G^h_{f_1,\dots,f_k}\in \Hom_K(P_h,Q_h).
\end{equation}
We know from Lemma \ref{lem_Mod_Delta} that $\delta_h$ behaves almost like an $h$-derivation. Thus, using the assumption that $L$ is an LDO of order $\leqslant k$, we obtain that the mapping \eqref{eq_multiderivation} is an $h$-derivation on every coordinate. In fact, for any $f,g,f_2,\dots, f_k \in A$
\begin{equation}
\begin{split}
G^h_{fg,f_2,\dots,f_k} & = \delta_h (fg)\delta_h (f_2)\dots \delta_h (f_k) L_h \\
\comment{\eqref{eq_delta}}& = (h(f)\delta_h(g)+h(g)\delta_h(f)+\delta_h(f)\delta_h(g))\delta_h (f_2)\dots \delta_h (f_k) L_h \\
\comment{L\textnormal{ is an LDO}} & = h(f)G^h_{g,f_2,\dots,f_k}+h(g)G^h_{f,f_2,\dots,f_k}.
\end{split}
\end{equation}
Since the $K$-algebra $A$ is commutative, we get that the mapping \eqref{eq_multiderivation} is symmetric. Hence, from the universal property of $(\Omega_h,d_h)$, there exists a $K$-linear map
\begin{equation}
\sigma_{L,h}:\left(\Omega_h\right)^{\odot k}\otimes_K P_h\to Q_h
\end{equation}
such that for any $f_1,\dots ,f_k\in A$ and $p\in P$
\begin{equation}
\sigma_{L,h}((d_h(f_1)\odot \dots\odot d_h(f_k))\otimes p_h)=\left( \delta_h (f_1)\dots \delta_h (f_k) L_h\right)(p).
\end{equation}
Notice that we moved $P_h$ into the argument using the tensor-hom adjunction. Finally, if $f_1,\dots,f_k\in I_h$, then clearly
\begin{equation}
\delta_h (f_1)\dots\delta_h (f_k) = 1_K\otimes f_1\cdots f_k.
\end{equation}
As a result, for any $f\in I_h$ and $p\in P$
\begin{equation}
\sigma_{L,h}((d_h(f))^{\odot k}\otimes p_h) = L_h(f^k p).
\end{equation}
The uniqueness of $\sigma_{L,h}$ follows from Proposition \ref{prop_Der_UniFactor}, which implies that the $K$-module $\Omega_h$ is generated by the image of $d_h$.
\end{proof}

In the consequence of Theorem \ref{thm_El_Symbol} we are able to define the principal symbol at any $K$-point and the ellipticity of a linear differential operator in a purely algebraic way.

\begin{definition*}
Let $L:P\to Q$ be an LDO of order $k$ and let $h\in |A|$. We will call $\sigma_{L,h}$ the \emph{principal symbol of $L$ at $h$}, where $\sigma_{L,h}$ is defined as in Theorem \ref{thm_El_Symbol}. We will say that \emph{$L$ is of order $k$ at $h$} if the map $\sigma_{L,h}\neq 0$. If $L$ is of order $k$ at $h$ and for every non-zero $\alpha\in \Omega_h$ the $K$-linear map $\sigma_{L,h,\alpha}:P_h\to Q_h$ given by the formula
\begin{equation}
\sigma_{L,h,\alpha}(p_h)=\sigma_{L,h}(\alpha^{\odot k}\otimes p_h)
\end{equation} 
is injective, then we will say that \emph{$L$ is elliptic at $h$}. Finally, we will say that \emph{$L$ is elliptic} if it is elliptic at $h$ for every $h\in |A|$.
\end{definition*}

\section{Elliptic LDOs on real affine varieties}
\label{sec_AffVar}

In this section we will show that the algebraic approach to elliptic LDOs from the previous section goes beyond the realm of smooth manifolds. We will show that for every real affine variety there is an elliptic linear differential operator acting on the algebra of regular functions on this variety. 

First, we notice that the algebra of regular functions on an affine variety has all $K$-points of a fixed form. More specifically, $K$-points of the algebra of regular functions correspond bijectively to ordinary points of variety, and they are precisely the evaluations.

\begin{prop}
\label{prop_Af_Kpoints}
Let $\mathbb{F}$ be a field and assume that $V$ is an affine variety embedded in $\mathbb{F}^n$. Then every $\mathbb{F}$-point of $\mathbb{F}[V]$ is of the form $\ev_z$ for some $z\in V$.
\end{prop}

\begin{proof}
Let $V$ be given by an ideal $I\subset\mathbb{F}[X_1,\dots,X_n]$ and assume that $h:\mathbb{F}[V]\to \mathbb{F}$ is an $\mathbb{F}$-algebra homomorphism. Since the $\mathbb{F}$-algebra $\mathbb{F}[V]$ is generated by $\overline{X}_1,\dots,\overline{X}_n$, we get that
\begin{equation}
h(\overline{f}) = f(h(\overline{X}_1),\dots, h(\overline{X}_n))
\end{equation}
for any $\overline{f}\in \mathbb{F}[V]$, where $f$ is some representative of $\overline{f}$. Thus $h = \ev_z$, where $z$ equals $(h(\overline{X}_1),\dots, h(\overline{X}_n))$. It is left to prove that $z\in V$. For any $g\in I$ we have that
\begin{equation}
g(z) = h(\overline{g}) = h(0) = 0.
\end{equation}
So $z\in V$.
\end{proof}

From Proposition \ref{prop_Af_Kpoints} and already mentioned fact from \cite[sec. 7.3]{Nestruev} (see the second paragraph preceding Lemma \ref{lem_El_Commute}), which characterises $\mathbb{R}$-points of $C^\infty(X)$, we obtain some curious observation. Even though characterisations of $K$-points on algebraic varieties and smooth manifolds are practically the same, the reasons why it is so are completely different. In the case of affine varieties, as we have seen above, the reasoning is purely algebraic. On the other hand, in the case of smooth manifolds the characterisation follows from the existence of a smooth function with compact level surfaces, which is rather a topological argument.

Given an affine variety, we have a natural surjective algebra homomorphism from the algebra of polynomials to the algebra of regular functions. Hence we are going to study the behaviour of $\delta$ and the principal symbol under the change of an algebra by a homomorphism.

\begin{prop}
\label{prop_Af_SymCom}
Assume that $S,T$ are two $K$-algebras and $\alpha:S\to T$ is a $K$-algebra homomorphism. If $\Phi:S\to S$ and $\Psi:T\to T$ are $K$-linear maps such that $\alpha\circ \Phi = \Psi\circ \alpha$, then for any $f\in S$
\begin{equation}
\label{eq_Af_PropCommute}
(\delta (\alpha(f))\Psi) \circ \alpha = \alpha \circ (\delta (f)\Phi).
\end{equation}
\end{prop}

\begin{proof}
Simply, for any $f,g\in S$ we have that
\begin{equation}
\begin{split}
(\delta (\alpha(f)) \Psi)(\alpha(g)) &= \Psi(\alpha(fg))-\alpha(f)\Psi(\alpha(g)) \\
\comment{\alpha\circ \Phi = \Psi \circ \alpha}&= \alpha(\Phi(fg))-\alpha(f\Phi(g)) \\
\comment{\delta(f)\textnormal{ acts on }\Phi}&= \alpha((\delta (f) \Phi)(g)).
\end{split}
\end{equation}
\end{proof}

\begin{lemma}
\label{lem_Af_SymCom}
Assume that $S,T$ are two $K$-algebras and $\alpha:S\to T$ is a $K$-algebra homomorphism. If $F:S\to S$ and $G:T\to T$ are $K$-linear maps such that $\alpha\circ F = G\circ \alpha$, then for any natural number $j$ and any $f_1,\dots,f_j\in S$
\begin{equation}
\label{eq_Af_LemmaCommute1}
(\delta (\alpha(f_1))\dots \delta (\alpha(f_j)) G) \circ \alpha = \alpha \circ (\delta (f_1)\dots\delta (f_j) F),
\end{equation}
and for any $h\in |T|$
\begin{equation}
\label{eq_Af_LemmaCommute2}
\delta_{h\circ\alpha}(f_1)\dots \delta_{h\circ\alpha}(f_j) F_{h\circ\alpha} = \delta_h(\alpha(f_1))\dots\delta_h(\alpha(f_j))G_h\circ \alpha.
\end{equation}
\end{lemma}

\begin{proof}
We will show the first part inductively. The first step of the induction is precisely Proposition \ref{prop_Af_SymCom}. So fix $f_0,\dots, f_j\in S$ and put
\begin{equation}
\Phi = \delta (f_1)\dots\delta (f_j) F, \quad \Psi = \delta (\alpha(f_1))\dots \delta (\alpha(f_j)) G.
\end{equation}
The inductive assumption states that $\alpha\circ \Phi = \Psi\circ \alpha$, so by using again Proposition \ref{prop_Af_SymCom} we have that
\begin{equation}
\begin{split}
(\delta (\alpha(f_0))\delta (\alpha(f_1))\dots \delta (\alpha(f_j)) G)\circ \alpha&=(\delta (\alpha(f_0))\Psi)\circ \alpha=\alpha\circ (\delta(f_0)\Phi)\\
&=\alpha\circ (\delta(f_0)\delta (f_1)\dots\delta (f_j) F).
\end{split}
\end{equation}
The second part of the lemma follows from the first one. From Lemma \ref{lem_El_Commute} for any $f_1,\dots,f_j\in S$ and any $K$-point $h\in|A|$ we have that
\begin{equation}
\begin{split}
\delta_{h\circ\alpha}(f_1)\dots \delta_{h\circ\alpha}(f_j) F_{h\circ\alpha} &= \ev_{h\circ\alpha}\circ (\delta(f_1)\dots\delta(f_j)F)\\
\comment{\ev_{h\circ\alpha}=\ev_h \circ\alpha}&= \ev_h\circ (\alpha\circ (\delta(f_1)\dots\delta(f_j)F))\\
\comment{\eqref{eq_Af_LemmaCommute1}}&= \ev_h\circ (\delta(\alpha(f_1))\dots \delta(\alpha(f_j))G\circ\alpha)\\
\comment{\textnormal{change of brackets}}&= (\ev_h\circ \delta(\alpha(f_1))\dots \delta(\alpha(f_j))G)\circ \alpha\\
\comment{\textnormal{Lemma \ref{lem_El_Commute}}}&= \delta_h(\alpha(f_1))\dots\delta_h(\alpha(f_j))G_h\circ \alpha.
\end{split}
\end{equation}
\end{proof}

We recall that for every $K$-algebra $A$ and $h\in|A|$ the $K$-module $A_h$ is naturally isomorphic with $K$. Thus if $L:A\to A$ is an LDO of order $\leqslant k$, then the principal symbol of $L$ at $h$
\begin{equation}
\sigma_{L,h}:\left(\Omega_h\right)^{\odot k}\otimes_K A_h\to A_h
\end{equation}
can be identified with the $K$-linear functional
\begin{equation}
\sigma_{L,h}:\left(\Omega_h\right)^{\odot k}\to K
\end{equation}
given by the formula
\begin{equation}
\label{eq_Af_SymoblAlg}
\sigma_{L,h}(d_h(f_1)\odot\dots\odot d_h(f_k)) = (\delta_h (f_1)\dots\delta_h(f_k)L_h)(1_A).
\end{equation}
We will use this identification when we consider LDOs acting on $K$-algebras.

\begin{lemma}
\label{lem_Af_Symbol}
Assume that $S,T$ are two $K$-algebras and $\pi:S\to T$ is a surjective homomorphism of $K$-algebras. If $L:S\to S$ is an LDO of order $\leqslant k$ and $L':T\to T$ is a $K$-linear map such that $\pi\circ L = L'\circ \pi$, then $L'$ is also an LDO of order $\leqslant k$ and for any $h\in |T|$
\begin{equation}
\label{eq_Af_LemmaSymbolEq}
\sigma_{L',h}\circ \left(\Omega_\pi\right)^{\odot k} = \sigma_{L,h\circ \pi},
\end{equation}
where $\Omega_\pi:\Omega_{h\circ\pi} \to\Omega_h$ is the induced map from $\pi$ by the functor $\Omega$ as in Proposition \ref{prop_Der_ComaCat}.
\end{lemma}

\begin{proof}
Since $\pi\circ L = L'\circ \pi$, by using Lemma \ref{lem_Af_SymCom} we obtain that for any $f_0,\dots,f_k\in S$
\begin{equation}
(\delta (\pi(f_0))\dots \delta (\pi(f_k)) L') \circ \pi = \pi \circ (\delta (f_0)\dots \delta (f_k)L ) = 0.
\end{equation}
Hence, from the surjectivity of $\pi$, we get that $\delta (g_0)\dots \delta (g_k) L' = 0$ for any $g_0,\dots,g_k\in T$. This means that $L'$ is an LDO of order $\leqslant k$. Now let $h\in |T|$ and $f_1,\dots,f_k\in S$. From the definition of $\Omega_\pi$ we have that $\Omega_\pi(d_{h\circ \pi}(f))=d_h(\pi(f))$ for every $f\in S$, and thus for any $f_1,\dots,f_k\in S$
\begin{equation}
\begin{split}
\sigma_{L',h} ((\Omega_\pi)^{\odot k})(d_{h\circ \pi} (f_1)\odot \dots \odot d_{h\circ \pi} (f_k)) &= \sigma_{L',h}(d_h(\pi(f_1))\odot\dots\odot d_h(\pi(f_k))) \\
\comment{\eqref{eq_Af_SymoblAlg}}&= (\delta_h (\pi(f_1))\dots \delta_h (\pi(f_k))L'_h)(1_T)\\
\comment{1_T = \pi(1_S)}&= (\delta_h (\pi(f_1))\dots \delta_h (\pi(f_k))L'_h)(\pi(1_S)) \\
\comment{\textnormal{Lemma }\ref{lem_Af_SymCom}}&= (\delta_{h\circ\pi}(f_1)\dots \delta_{h\circ\pi}(f_k)L_{h\circ \pi})(1_S)\\
\comment{\eqref{eq_Af_SymoblAlg}}&= \sigma_{L,h\circ \pi}(d_{h\circ \pi} (f_1)\odot \dots \odot d_{h\circ \pi} (f_k)).
\end{split}
\end{equation} 
\end{proof}

As a result, we get that the ellipticity is preserved by a surjective homomorphism of $K$-algebras in the following sense.

\begin{theorem}
\label{thm_Af_InducElip}
Assume that $S,T$ are two $K$-algebras, $\pi:S\to T$ is a surjective homomorphism of $K$-algebras and $L:S\to S$ is an LDO of order $k$ at $h\circ\pi$ and elliptic at $h\circ\pi$ for every $h\in |T|$. If $L':T\to T$ is $K$-linear and such that $\pi\circ L = L'\circ \pi$, then $L'$ is an elliptic LDO of order $k$.
\end{theorem}

\begin{proof}
From Lemma \ref{lem_Af_Symbol} we have that $L'$ is an LDO of order $\leqslant k$ and for every $h\in|T|$
\begin{equation}
\label{eq_Af_ThmSymSur}
\sigma_{L',h}\circ \left(\Omega_\pi\right)^{\odot k} = \sigma_{L,h\circ \pi}.
\end{equation}
Thus we just need to show that $\sigma_{L',h}(\alpha^{\odot k})\neq 0$ for any non-zero $\alpha\in\Omega_h$. Proposition \ref{prop_Der_SurPre} states that the functor $\Omega$ preserves surjectivity. So we have that any $\alpha\in\Omega_h$ is of the form $\Omega_\pi(\beta)$ for some $\beta\in\Omega_{h\circ\pi}$. Thus we have that
\begin{equation}
\sigma_{L',h}(\alpha^{\odot k}) = \sigma_{L',h}((\Omega_{\pi}(\beta))^{\odot k})=\sigma_{L,h\circ \pi}(\beta^{\odot k}),
\end{equation}
and this is non-zero for non-zero $\alpha$ because of the assumption about $L$.
\end{proof}

Now, let $V$ be a real affine variety defined by some ideal $I=(f_1,\dots,f_k)\subset\mathbb{R}[X_1,\dots,X_n]$. Let $\mathbb{R}[V]=\mathbb{R}[X_1,\dots,X_n]/I$ denote the $\mathbb{R}$-algebra of regular functions on $V$ and let $\pi:\mathbb{R}[X_1,\dots,X_n]\to\mathbb{R}[V]$ be the canonical projection. We know from Proposition \ref{prop_Af_Kpoints} that $\mathbb{R}$-points of $\mathbb{R}[X_1,\dots,X_n]$ and $\mathbb{R}[V]$ correspond bijectively to $\mathbb{R}^n$ and $V$ respectively. Thus we can talk about the ellipticity at ordinary points. By applying Theorem \ref{thm_Af_InducElip} to the homomorphism $\pi$ we obtain a simple tool for constructing elliptic linear differential operators on real affine varieties.

\begin{corollary}
\label{cor_Af_Var}
If $L:\mathbb{R}[X_1,\dots,X_n]\to\mathbb{R}[X_1,\dots,X_n]$ is an LDO of order $k$ at $z$ and elliptic at $z$ for every $z\in V$ and such that $L(I)\subset I$, then the induced map $L':\mathbb{R}[V]\to\mathbb{R}[V]$ is an elliptic linear differential operator of order $k$.
\end{corollary}

\begin{proof}
Since $L(I)\subset I$, we have that $L$ induces an $\mathbb{R}$-linear map $L':\mathbb{R}[V]\to\mathbb{R}[V]$ such that $L'\circ\pi = \pi\circ L$. The rest follows directly from Theorem \ref{thm_Af_InducElip}.
\end{proof}

\begin{corollary}
\label{cor_Af_ellipticOnVariety}
Every real affine variety admits an elliptic linear differential operator acting on its algebra of regular functions.
\end{corollary}

\begin{proof}
For the sake of readability we constrain ourselves to varieties on the plane. The general case can be proved similarly. So let $I\in \mathbb{R}[X,Y]$ be an ideal generated by some polynomials $f_1,\dots,f_k$. For each $f_i$ we define natural numbers $s_i$ and $t_i$ as the degrees of $f_i$ in variables $X$ and $Y$ respectively. Let $s$ be the minimal even number greater than all $s_i$ and analogously let $t$ be the minimal even number greater than all $t_i$. We define the map $\Delta_I:\mathbb{R}[X,Y]\to\mathbb{R}[X,Y]$ as
\begin{equation}
\label{eq_Af_CorHiperLaplace}
\Delta_I = \frac{\partial^s}{\partial X^s}+\frac{\partial^t}{\partial Y^t}.
\end{equation}
Clearly, $\Delta_I$ is an elliptic LDO, and the numbers $s$ and $t$ were tailored in such a way that $\Delta_I(I)=0$. In particular, $\Delta_I(I)\subset I$. So, using Corollary \ref{cor_Af_Var}, we get that the induced map $\Delta_I':\mathbb{R}[V]\to\mathbb{R}[V]$ is an elliptic linear differential operator.
\end{proof}

For a concrete example, take the ideal $I=(Y^3-X^2)\subset\mathbb{R}[X,Y]$ and let $V$ be the real affine variety generated by $I$. According to the above description, we consider $\Delta_I:\mathbb{R}[X,Y]\to\mathbb{R}[X,Y]$ by the formula
\begin{equation}
\Delta_I = \frac{\partial^4}{\partial X^4}+\frac{\partial^4}{\partial Y^4},
\end{equation}
and this LDO induces an elliptic LDO $\Delta'_I:\mathbb{R}[V]\to\mathbb{R}[V]$ of order $4$. 

The construction of $\Delta'_I$, presented in the proof of Corollary \ref{cor_Af_ellipticOnVariety}, is far from pleasing. $\Delta'_I$ is just a "restriction" of some operator of the form \eqref{eq_Af_CorHiperLaplace} defined on the ambient space. What is even worse, it depends on the generator $I$ of $V$. However, it is interesting whether there is a way to associate an elliptic LDO to a real affine variety in some natural manner.

\appendix
\section{Appendix: Universal $\phi$-derivations and Kähler differentials}
\label{app_Kähler}

One familiar with the theory of Kähler differentials may notice that most constructions from Sections \ref{sec_Deriv} and \ref{sec_Models} are just modified versions of objects appearing in the theory of Kähler differentials. In this appendix we will make clear why it is so.
 
If $A$ is a $K$-algebra and $P$ is an $A$-module, then the set of all derivations with values in $P$ is denoted by $\Der_K(A,P)$. The universal derivation of $A$, which is the same as the universal $\Id_A$-derivation in the language from Section \ref{sec_Deriv}, is called the \emph{Kähler differential of $A$} and it is denoted by $(\Omega_{A/K},d_{A/K})$.

Let $A,B$ be two $K$-algebras and let $\phi:A\to B$ be a $K$-algebra homomorphism. Because of $\phi$, we have the \emph{restriction of scalars by $\phi$} which endows any $B$-module with an $A$-module structure. More specifically, if $M$ is a $B$-module, then $M_\phi$ denotes the set $M$ with an $A$-module structure such that for any $a\in A$ and $m\in M_{\phi}$
\begin{equation}
am = \phi(a)m.
\end{equation}
As a result, $\phi$-derivations with values in $M$ are the same as derivations with values in $M_\phi$. In fact, if $D:A\to M_\phi$ is a derivation, then for any $f,g\in A$
\begin{equation}
D(fg)=gD(f)+fD(g)=\phi(g)D(f)+\phi(g)D(f).
\end{equation}

The restriction of scalars by $\phi$ is a functor from the category of $B$-modules to the category of $A$-modules. It has a left adjoint, called the \emph{extension of scalars by $\phi$}, which is just the tensor multiplication by $B$ (see \cite[p. 277]{Bourbaki} for instance). This means that for every $A$-module $P$ and every $B$-module $M$ there is a natural isomorphism
\begin{equation}
\Hom_A(P,M_\phi)\cong\Hom_B(B\otimes_A P, M).
\end{equation}
As a result, for any $B$-module $M$ we have that
\begin{equation}
\begin{split}
\Der_K(\phi,M)&\cong\Der_K(A,M_\phi)\\
&\cong\Hom_A(\Omega_{A/K},M_\phi)\\
&\cong\Hom_B(B\otimes_A\Omega_{A/K},M).
\end{split}
\end{equation}
So the $B$-module $B\otimes_A \Omega_{A/K}$ with the function $d'_\phi: A\to B \otimes_A \Omega_{A/K}$ which maps $f$ from $A$ to $1_B\otimes d_{A/K}(f)$ in $B\otimes_A\Omega_{A/K}$ is a representing object for the functor $\Der_K(\phi,\cdot)$.

\begin{prop}
\label{prop_Apppen}
If $A,B$ are $K$-algebras and $\phi:A\to B$ is a $K$-algebra homomorphism, then the pair $(B\otimes_A \Omega_{A/K}, d'_\phi)$ is a universal $\phi$-derivation.
\end{prop}

Hence most of results from Sections \ref{sec_Deriv} and \ref{sec_Models} can be obtained by combining Proposition \ref{prop_Apppen} and certain properties of Kähler differentials. However, such approach requires additional categorial tools about representing objects and some specific properties of Kähler differentials which are scattered across several sources. Thus, for the sake of clarity and readability of Sections \ref{sec_Examples}, \ref{sec_Elliptic} and \ref{sec_AffVar}, we had chosen a simpler path and done a concise analysis of $\phi$-derivations on its own.

\bigskip
\noindent\textbf{Acknowledgements:} First of all, I am most grateful to Antoni Pierzchalski for spending his time to listen about the main ideas from this paper and for many useful comments. I am also grateful to Sameer Kailasa for showing me the proof of Proposition \ref{prop_Af_Kpoints} through online discussion. Last but not least, I would like to thank my brother Radosław Kapka for careful reading the final draft of this paper and for numerous hints regarding writing in English.

\end{document}